\documentclass{amsart}

\pdfoutput=1

\usepackage{amssymb,amsthm,amsmath,tensor}
\usepackage[pdftex]{graphicx}

\usepackage{enumerate}

\theoremstyle{theorem}
\newtheorem{theorem}{Theorem}[section]
\newtheorem{proposition}[theorem]{Proposition}
\newtheorem{lemma}[theorem]{Lemma}
\newtheorem{corollary}[theorem]{Corollary}
 
\theoremstyle{definition}
\newtheorem{definition}[theorem]{Definition}

\newtheorem{remark}[theorem]{Remark}
 
\def\im{\mathop{\rm Im}\nolimits}

\newcommand{\m}{\mathfrak{m}}

\def\im{\mathop{{\mathrm{im}}}}
\def\ker{\mathop{{\mathrm{ker}}}}

\def\Ann{{{\mathrm{Ann}}}}

\def\ldiv{\mathbin{{\!\smallsetminus\!}}}
\def\rdiv{\mathbin{{\!\reflectbox{$\smallsetminus$}\!}}}

\begin{document}
	\title{Local loop near-rings}
	\author{Damir Franeti\v{c}}
	\address{Univerza v Ljubljani, Fakulteta za ra\v cunalni\v stvo in informatiko, Ve\v{c}na pot 113, 1000 Ljubljana, Slovenia}
	\email{damir.franetic@fri.uni-lj.si}
	\begin{abstract}
		We study loop near-rings, a generalization of near-rings, where the additive structure is not necessarily associative. 
		We introduce local loop near-rings and prove a useful detection principle for localness.
	\end{abstract}

\maketitle

\section*{Introduction}
	This paper evolved from a number of algebraic results that proved to be useful in the study of decompositions of H- and coH-spaces by the 
	author and his advisor in~\cite{FraPav} and~\cite{FraPav2}. Generalizing the notion of localness from rings to loop near-rings, we 
	were able to prove powerful uniqueness-of-decompostion results for H- and coH-spaces, which are analogous to the classical 
	Krull--Schmidt--Remak--Azumaya theorem for modules.
	
	A near-ring is a generalization of the notion of a ring, where one does not assume the addition to be commutative, and only one 
	distributivity law holds. This is a well-studied algebraic structure, see~\cite{Pilz},~\cite{Meldrum},~\cite{Clay}. Loop near-rings were introduced 
	in~\cite{Ramakotaiah} as a generalization of near-rings. In a loop near-ring $N$ one does not even require the addition to be 
	associative, instead $N$ is only assumed to be an algebraic loop under addition. To justify the study of such an obscure algebraic 
	structure, we note that homotopy endomorphisms of connected H-spaces are examples of genuine loop near-rings, which are often not 
	near-rings~\cite[Examples 1.4 and 1.5]{FraPav2}. 
	
	The paper is divided into three sections. In Section~\ref{sect:intro} we recall the definitions of loops, loop near-rings, their 
	modules, and module homomorphisms. Relevant substructures are then defined naturally as kernels and images of those homomorphisms. There 
	are no new results, we do, however, reprove several known facts in a novel and concise manner. In Section~\ref{sect:jac} two 
	(of the several) possible generalizations of the Jacobson radical to loop near-rings are defined. We introduce quasiregular elements 
	and show that in certain important cases both Jacobson radical-like objects coincide with the largest quasiregular ideal. Finally, 
	in Section~\ref{sect:local}, local loop near-rings are introduced, and it is shown that many known properties of local rings also 
	hold in the loop near-ring setting. 
\section{Loops and loop near-rings}
\label{sect:intro}
	A loop is a generalization of the notion of a group. Associativity requirement is dropped from the definition, but one still 
	requires the existence of an identity element and replaces the existence of inverses by existence of unique solutions to 
	certain equations. A loop near-ring is a generalization of a ring. Two requirements are omitted from the definition of a ring: 
	commutativity and associativity of addition, and right or left distributivity. Nevertheless, a surprising amount of 
	common concepts and theorems from ring theory generalizes to this setting. Loop near-rings were first introduced by 
	Ramakotaiah in~\cite{Ramakotaiah}. We recall the definitions and state relevant results.
	
	\begin{definition}
		An algebraic structure $(G,+)$, where $+$ denotes a binary operation on the set $G$, is a {\em quasigroup} if, for all
		$a,b \in G$, the equations $a+x=b$ and $y+a=b$ have unique solutions $x, y \in G$. If a quasigroup $(G,+)$ has a 
		two-sided zero, i.e. an element $0 \in G$ such that $0+a=a+0=a$ for all $a \in G$, we call $G$ a {\em loop}.
	\end{definition}
	Every group is a loop, and a loop is essentially a `non-associative group'. Existence of unique solutions to the two 
	equations implies that left and right cancellation laws hold in a loop. The unique solution 
	of the equation $a+x=b$ will be denoted by $x = a \ldiv b$, and the unique solution of the equation $y+a=b$ by $y= b \rdiv a$.  
	The operations $\ldiv$ and $\rdiv$ are called the {\em left} and the {\em right difference} respectively.
	
	There are two kinds of substructures that will interest us. A subset $I$ of a loop $H$ is called a {\em subloop} if it is closed under the operations 
	$+$, $\ldiv\,$, and $\rdiv$ on $H$. 
	Notation $I \le H$ will stand for `$I$ is a subloop of $H$'. The definition of a normal subloop is more complicated due to lack of associativity. Given a 
	loop $G$ a subloop $K \le G$ is a {\em normal subloop} if for all $a, b \in G$ we have
	\[
		a+K = K+a \textrm{, }(a+b)+K = a+(b+K) \textrm{ and } (K+a)+b = K+(a+b) \textrm{. }
	\]
	Notation $K \unlhd G$ will stand for `$K$ is a normal subloop of $G$'. Whenever $K$ is a normal subloop of $G$, the quotient $G/K$ admits a natural 
	loop structure determined by $(a+K)+(b+K) := (a+b)+K$.
	
	In the present paper we prefer to characterize substructures naturally (in the sense of category theory). A map of loops 
	$\phi \colon G \to H$ is a {\em loop homomorphism} if $\phi(a+b) = \phi(a)+\phi(b)$ holds for all $a, b \in G$. Since $\phi(0) + \phi(0) = \phi(0)$, 
	cancellation in $H$ gives $\phi(0)=0$. Similarly; $\phi(a \ldiv b) = \phi(a) \ldiv \phi(b)$ and $\phi(a \rdiv b) = \phi(a) \rdiv \phi(b)$. 
	The {\em category of loops} has loops as objects and loop homomorphisms as morphisms. It is a category with a zero object, namely the trivial loop $0$ 
	consisting of the zero element only. Hence, there is the {\em zero homomorphism} $0 \colon G \to H$ between any two loops $G$ and $H$ mapping every 
	element of $G$ to $0 \in H$. The {\em kernel} of a loop homomorphism $\phi \colon G \to H$ is the preimage of $0 \in H$, i.e. $\ker \phi = \phi^{-1}(0)$.
	This $\ker \phi$ is the equalizer of $\phi$ and $0\colon G \to H$, so $\ker \phi$ is in fact a category-theoretic kernel. The {\em image} of a loop 
	homomorphism is the set $\im \phi = \phi(G)$. Observe that normal subloops are precisely kernels, while subloops are precisely images. Specifically, 
	$K \unlhd G$ if and only if $K$ is the kernel of some loop homomorphism, and $I \le H$ if and only if $I$ is the image of some loop homomorphism. This kind of 
	characterization of substructures will be used as the defining property later in this paper. It has the advantage of avoiding (often complicated) 
	element-by-element defining conditions, and streamlines many proofs. 
	See~\cite[Chapter IV]{Bruck} for a detailed treatment of loops, their homomorphisms, and corresponding substructures.
	
	Recall that $(S,\cdot)$ is a {\em semigroup} if the binary operation $\cdot$ on $S$ is associative.
	\begin{definition}
		A {\em loop near-ring} $N$ is an algebraic structure $(N,+,\cdot)$ such that:
		\begin{itemize}
			\item $(N,+)$ is a loop, 
			\item $(N, \cdot)$ is semigroup, 
		\end{itemize}
		and multiplication $\cdot$ is either left or right distributive over addition $+$. If we have: 
		\begin{itemize}
			\item $m(n_1+n_1) = mn_1 + mn_2$ for all $m, n_1, n_2 \in N$, we call $N$ a {\em left} loop near-ring, 
			\item $(m_1+m_2)n = m_1n+m_2n$ for all $m_1, m_2, n \in N$, we call $N$ a {\em right} loop near-ring.
		\end{itemize} 
		If $(N, +)$ is a group, $(N,+, \cdot)$ is a {\em near-ring}.
	\end{definition}
	A common example of a right near-ring is the near-ring of {\em all} functions $f \colon G \to G$ of a group 
	$(G, +)$, commonly denoted by $M(G)$. When $G$ is merely a loop, $M(G)$ is a loop near-ring~\cite[Example 1.2]{Ramakotaiah}. 
	Restricting to only those functions $f \colon G \to G$ for which $f(0) = 0$ we obtain $M_0(G)$ -- the {\em zero-symmetric} 
	part of $M(G)$. 
		
	We will restrict our discussion to right loop near-rings. 
	Right distributivity in $N$ implies that right multiplication  
	$\phi_n \colon N \to N$, $m \mapsto mn$, by $n$ is an endomorphism of the loop $(N,+)$, and it follows that 
	$(m_1 \ldiv m_2)n = m_1 n \ldiv m_2 n$, $(m_1 \rdiv m_2)n = m_1 n \rdiv m_2 n$, and $0n = 0$ for all $m_1, m_2, n \in N$.
	
	Note that $n0 \neq 0$ in general. However, for arbitrary $n \in N$, the solution $y$ of the equation $n = y + n0$ does 
	satisfy $y0=0$, since $n0 = (y+n0)0 = y0+n0$. Therefore $N = N_0 + N_c$, where 
	$N_0 = \{ y \in N : y0 = 0\}$, $N_c = N0 = \{ n0 : n \in N\}$, and $N_0 \cap N_c = 0$.
	We call $N_0$ {\em the zero-symmetric part} and $N_c$ {\em the constant part} of $N$, respectively. Also, a loop 
	near-ring $N$ will be called {\em zero-symmetric} if $N = N_0$, i.e. $n0 = 0$ holds for all $n \in N$.

	A loop near-ring $N$ is {\em unital} if there is an element $1 \in N$ (called {\em the identity}), such that $1n = n1 = n$ 
	for all $n \in N$. An element $u \in N$ in a unital loop near-ring $N$ is called a {\em unit} (or {\em invertible}) if there is a 
	$u^{-1} \in N$ ({\em the inverse} of $u$), such that $u u^{-1} = u^{-1} u = 1$. The group of all units of $N$ 
	will be denoted by $U(N)$. A loop near-ring $N$ is a {\em loop near-field} if $U(N) = N \setminus\{ 0\}$.
	\begin{definition}
		A loop $G$ is a {\em left module} over a (right) loop near-ring $N$ if there is an action
		\[
			N \times G \to G \textrm{, } (n,a) \mapsto na
		\]
		such that $m(na) = (mn)a$, and $(m+n)a = ma+na$ hold for all $a \in G$ and $m,n \in N$.  If $N$ is unital, we also 
		require the action to be unital, i.e. $1a = a$ for all $a \in G$. To emphasize that $G$ is a left $N$-module, we will often 
		write $\tensor*[_N]{G}{}$.
		
		A loop $G$ is a {\em right module} over a (right) loop near-ring $M$ if there is an action 
		\[
			G \times M \to G \textrm{, } (a, m) \mapsto am
		\]
		such that $a(nm) = (an)m$, and $(a+b)n = an + bn$ hold for all $a,b \in G$ and $m,n \in M$. If $1 \in M$, we also 
		require $a1=a$ for all $a \in G$. We will denote right $M$-modules by $G_M$. 
		
		A loop $G$ is a {\em $(N,M)$-bimodule} if $G$ is both a left $N$-module and a right $M$-module, and $(na)m = n(am)$ holds for 
		all $a \in G$, $n \in N$, $m \in M$. Notation: $\tensor*[_N]{G}{_M}$.
		\label{def:modules}
	\end{definition}
	Every loop near-ring $N$ is an $(N,N)$-bimodule with the action defined by loop near-ring multiplication. 
	Also, we view every left $N$-module $G$ as an $(N,0)$-bimodule, where $0$ is the trivial loop near-ring, and every right $M$-module 
	$G$ as a $(0,M)$-bimodule. Both trivial actions are defined by $a0=0$ and $0a = 0$ for all $a \in G$. Therefore, every loop $G$ is, 
	trivially, a $(0,0)$-bimodule. 
	
	In the next definition we define substructures in an unconventional, but natural way.
	\begin{definition}
		For left $N$-modules $G$ and $H$, we call a map $\phi \colon G \to H$ a {\em homomorphism of left $N$-modules} if 
		$\phi(a+b) = \phi(a) + \phi(b)$ and  $\phi(na) = n \,\phi(a)$ holds for all $a, b \in G$ and all $n\in N$. 
		{\em Homomorphisms of right $M$-modules} are defined analogously. A map is a {\em homomorphism of $(N,M)$-bimodules} if it is a 
		left $N$-module and a right $M$-module homomorphism simultaneously.
		
		If $K = \ker \phi \subseteq G$ is the kernel of a homomorphism of left $N$-modules $\phi \colon G \to H$, we call $K$ a 
		{\em left $N$-submodule} and write $K \tensor*[_N]{\unlhd}{} G$. If $I = \im \phi \subseteq H$ is the image of a homomorphism 
		$\phi$, we call $I$ a {\em left $N$-subloop} and write $I \tensor*[_N]{\le}{} H$. We define
		{\em right $M$-submodules} $K \unlhd_M G$ and {\em right $M$-subloops} $I \le_M H$ analogously, i.e. as kernels and images of 
		right $M$-module homomorphisms. It should be clear, what 
		is meant by {\em $(N,M)$-submodule}, $K \tensor*[_N]{\unlhd}{_M} G$, and {\em $(N,M)$-subloop}, $I \tensor*[_N]{\le}{_M} H$. 
		
		A subset $J \subseteq N$ in a loop near-ring $N$ is a {\em left ideal} if $J$ is a left $N$-submodule in $\tensor*[_N]{N}{}$, 
		a {\em right ideal} if $J$ is a right $N$-submodule in $N_N$, and an {\em ideal} if $J$ is an $(N,N)$-submodule in 
		$\tensor*[_N]{N}{_N}$. 
		\label{def:substructures}
	\end{definition}
	Note that left $N$-subloops are left $N$-modules on their own right, while left $N$-submodules are {\em not} 
	left $N$-modules unless $N$ is zero-symmetric. For if $n0 \neq 0$ for some $n \in N$ then $n 0 \notin \ker \phi$ since 
	$\phi(n 0) = n \, \phi(0) = n 0$. Here $\phi \colon G \to H$ is a left $N$-module homomorphism and $0$ denotes the zero in $N$, 
	$G$, or $H$ as required. Right structures exhibit nicer behavior: right $M$-submodules and right $M$-subloops {\em are} right 
	$M$-modules.
	\begin{remark}
		A word of caution regarding naming conventions. In the near-ring setting Pilz~\cite{Pilz} calls our left $N$-modules $N$-groups, our 
		left $N$-submodules are called ideals, while our left $N$-subloops are (for our convenience) renamed as $N$-subgroups. On the other 
		hand, Meldrum~\cite{Meldrum} and Clay~\cite{Clay} use the same name as we do for left $N$-modules, while our left $N$-subgroups 
		are called $N$-submodules, and our left $N$-submodules are called ($N$-)ideals. It seems that right structures have not yet been extensively 
		studied, but Clay~\cite[Definition 13.2]{Clay} does define them and calls our right $M$-modules $M$-comodules. Admittedly, our naming 
		convention is a little confusing in view of the fact described above. To our defense, let us just say that the confusion disappears if $N$ is 
		zero-symmetric.
	\end{remark} 
	\begin{remark}
		A map $\phi \colon N \to M$ is a {\em homomorphism of loop near-rings} if $\phi(n_1 + n_2) = \phi(n_1)+\phi(n_2)$ 
		and $\phi(n_1 n_2) = \phi(n_1)\phi(n_2)$ holds for all $n_1, n_2 \in N$. If $N$ and $M$ are unital, we add the requirement 
		$\phi(1) = 1$. Kernels of such homomorphisms are ideals in the sense of Definition~\ref{def:substructures}, since $M$ can be viewed as an 
		$(N,N)$-bimodule with the two actions defined by $n \cdot m := \phi(n)\, m$ and $m \cdot n := m \, \phi(n)$ for $n \in N$, 
		$m \in M$.
	\end{remark}
	Since $\phi(0) = 0$ for any loop homomorphism $\phi \colon G \to H$, we can view $\phi$ as a homomorphism of $(0,0)$-bimodules. 
	Hence, a normal subloop $K \unlhd G$ is the same as a 
	$(0,0)$-submodule $K \tensor*[_0]{\unlhd}{_0} G$, and a subloop $I \le H$ is the same as a $(0,0)$-subloop 
	$I \tensor*[_0]{\le}{_0}H$. Also: $K \unlhd_N G$ $\Leftrightarrow$  $K \tensor*[_0]{\unlhd}{_N} G$, 
	$I \le_N H$ $\Leftrightarrow$ $I \tensor*[_0]{\le}{_N} H$, $K \tensor*[_M]{\unlhd}{} G$ $\Leftrightarrow$ 
	$K \tensor*[_M]{\unlhd}{_0} G$, and $I \tensor*[_M]{\le}{} H$ $\Leftrightarrow$ $I \tensor*[_M]{\le}{_0} H$.

	Our definition of substructures is of little use when one wants to do element-by-element computations. In the next proposition we translate 
	Definition~\ref{def:substructures} into conventional element-wise defining conditions. The proof is a routine exercise, so we omit it.
	\begin{proposition}
		Let $N$ and $M$ be loop near-rings and $G = \tensor*[_N]{G}{_M}$ an $(N,M)$-bimodule. The following assertions hold:
		\begin{enumerate}[(a)]
			\item $K \subseteq G$ is a left $N$-submodule if and only if $K$ is a normal subloop in $(G,+)$ and
				$n(a+k) + K = na + K$ holds for all $k \in K$, $a \in G$ and $n \in N$. 
			\item $K \subseteq G$ is a right $M$-submodule if and only if $K$ is a normal subloop in $(G,+)$ and $MK \subseteq K$.
			\item $I \subseteq G$ is a left $N$-subloop if and only if $I$ is a subloop in $(G,+)$ and $NI \subseteq I$.
			\item $I \subseteq G$ is a right $M$-subloop if anf only if $I$ is a subloop in $(G,+)$ and $IM \subseteq I$.
			\item If $N$ is zero-symmetric, then every left $N$-submodule in $G$ is also an $N$-subloop.
		\end{enumerate}
		\label{prop:substructures}
	\end{proposition}
	Absence of left distributivity is the reason for the lack of symmetry between the element-wise characterizations of left $N$-submodules and right 
	$M$-submodules in Proposition~\ref{prop:substructures}. Over {\em right} loop near-rings {\em left} modules will play a pivotal role in 
	the radical theory. Over {\em left} loop near-rings the roles of left and right modules are reversed.
	
	For an $(N,M)$-submodule $K \tensor*[_N]{\unlhd}{_M} G$ the loop $G/K$ admits a natural $(N,M)$-bimodule structure with the two 
	actions defined by $n(a+K) := na + K$ and $(a+K)m := am + K$. Also, for an ideal $J \unlhd N$, the quotient $N/J$ becomes a loop 
	near-ring with multiplication defined by $(n+J)(m+J) := nm + J$. Again, we deliberately omit both verifications.
	
	If $\phi \colon G \to H$ is a homomorphism of $(N,M)$-bimodules $G$ and $H$, then the preimage $\phi^{-1}(K)$ of an $(N,M)$-submodule 
	$K \tensor*[_N]{\unlhd}{_M} H$ is a $(N,M)$-submodule in $G$, since $\phi^{-1}(K)$ is the kernel of the composition  
	$G \to H \twoheadrightarrow H/K$. Similarly, the preimage $\phi^{-1}(I)$ of an $(N,M)$-subloop  
	$I \tensor*[_N]{\le}{_M} H$, is an $(N,M)$-subloop in $G$. (This is a routine verification using 
	Proposition~\ref{prop:substructures}.)  
	 
	We obtain the following `correspondence theorem'.
	\begin{proposition}
		\label{prop:correspondence}
		Let $\phi \colon G \to H$ be a homomorphism of $(N,M)$-bimodules. Then $\phi$ induces an isomorphism 
		$\bar{\phi} \colon G/\ker\phi \to \im \phi$, which determines a bijective correspondence between $(N,M)$-subloops or 
		$(N,M)$-submodules in $\im \phi$ and those $(N,M)$-subloops or $(N,M)$-submodules in $G$, which contain $\ker \phi$.
	\end{proposition}
	\begin{proposition}
		\label{prop:K+I}
		Let $G$ be an $(N,M)$-bimodule, $K$ an $(N,M)$-submodule in $G$, and $I$ an $(N,M)$-subloop in $G$. Then $K+I = I+K$ and 
		$K+I$ is an $(N,M)$-subloop in $G$. 
	\end{proposition}
	\begin{proof}
		Write $K = \ker \phi$ for some homomorphism $\phi \colon G \to H$. For all $k \in K$ and $i \in I$ we have 
		$\phi(k+i) = \phi(k) + \phi(i) = \phi(i) \in \phi(I)$, hence $K+I \subseteq \phi^{-1}(\phi(I))$. Also, for any 
		$a \in \phi^{-1}(\phi(I))$ there is an $i \in I$, such that $\phi(i) = \phi(a)$. If $y$ is the solution of 
		the equation $y + i = a$, then $\phi(y) = 0$, hence $y \in K$, and it follows that $a \in K + I$. 
		Conclusion: $K + I = \phi^{-1}(\phi(I))$. The proof that $I + K$ equals $\phi^{-1}(\phi(I))$ is completely analogous. 
		Note that any preimage of an $(N,M)$-subloop is an $(N,M)$-subloop by Proposition~\ref{prop:substructures}, so 
		$\phi^{-1}(\phi(I))$ is an $(N,M)$-subloop in $G$.		
	\end{proof}
	\begin{lemma}
		\label{lem:intersub}
		The intersection of an arbitrary family of $(N,M)$-submodules or $(N,M)$-subloops in an $(N,M)$-bimodule $G$ is an 
		$(N,M)$-submodule or an $(N,M)$-subloop, respectively. 
	\end{lemma}
	\begin{proof}
		Let $K_i = \ker(\phi_i \colon G \to H_i)$ be $(N,M)$-submodules in $G$. Denote by 
		$\Delta \colon G \to \prod_i (G)_i$ the diagonal, and by 
		$\prod_i \phi_i \colon \prod_i (G)_i \to \prod_i H_i$ the product homomorphism. Then 
		$\bigcap_i K_i = \ker \phi$, where $\phi$ is the composite
		\[
			\phi \colon G \xrightarrow{\Delta} \prod_i (G)_i \xrightarrow{\prod_i \phi_i} \prod_i H_i
			\textrm{, } \ a \mapsto \left(\phi_i(a)\right)_i \textrm{, }
		\]
		hence $\bigcap_i K_i$ is an $(N,M)$-bimodule. 
		
		For $(N,M)$-subloops $I_k = \im(\psi_k \colon F_k \to G)$ in the $(N,M)$-bimodule $G$ we have 
		$\bigcap_k I_k = \Delta^{-1}(\im \prod_k \psi_k)$. 	
	\end{proof}
	\begin{definition}
		Let $G$ be a left $N$-module. For subsets $A,B \subseteq G$ define
		\[
			(A:B) = \tensor*[_N]{(A:B)}{} := \{ n \in N : nB \subseteq A \} \textrm{. }
		\]
		For singletons $A = \{a \}$ or $B = \{b \}$ we will write $(a:B)$ or $(A:b)$, respectively. The {\em annihilator} of $B$ in 
		$N$ is $\Ann(B) := {(0:B)} = \{ n \in N : nB = 0\}$.
	\end{definition}
	\begin{lemma}
		Let $G$ be a left $N$-module, and $B \subseteq G$ an arbitrary subset. 
		If $A$ is a subloop (normal subloop, left $N$-subloop, left $N$-submodule) in $G$, then $(A:B)$ is a 
		subloop (normal subloop, left $N$-subloop, left ideal) in $\tensor*[_N]{N}{}$.
	\end{lemma}
	\begin{proof}
		Note that $(A: \bigcup_i B_i) = \bigcap_i (A : B_i)$. For $b \in B$ we have the left $N$-module 
		homomorphism $\phi_b \colon N \to G$, $n \mapsto nb$, and $(A:b) = \phi_b^{-1}(A)$ holds. Now $(A:B) = \bigcap_{b\in B} (A:b)$, 
		and Proposition~\ref{prop:correspondence} and Lemma~\ref{lem:intersub} imply that $(A:B)$ in $N$ is a substructure 
		of the same kind as $A$ is in $G$.  
	\end{proof}
	It is easy to check that $(K:G)$ is an ideal in $N$ for any left $N$-submodule $K \tensor*[_N]{\unlhd}{} G$, and we obtain the following 
	corollary. 
	\begin{corollary}
		\label{cor:AnnG}
		For any left $N$-module $G$ and any $a \in G$, $\Ann(a)$ is a left ideal in $N$, and $\Ann(G)$ is a (two-sided) ideal in $N$.
	\end{corollary}
\section{Jacobson radicals, quasiregularity, and local homomorphisms}	
	\label{sect:jac}
	The Jacobson radical $J(R)$ of a ring $R$ is defined as the intersection of all maximal left ideals in $R$ or annihilators of all simple 
	left $R$-modules. If $R$ is unital, then $J(R)$ is also characterized as the largest quasiregular ideal in $R$.
	There are several possible generalizations of simplicity to left modules over (right) loop near-rings $N$, each of which comes with its 
	corresponding `Jacobson radical'. Of course, all of these coincide when $N$ is a ring. We recall two of them below, which will 
	suffice for our purposes. In order to have a well-behaved $J$-radical theory, we restrict our attention to unital, zero-symmetric 
	loop near-rings $N$.
	
	A left, right or two-sided ideal $K \unlhd N$ is {\em maximal} if $K \neq N$ and there is no ideal of the same kind between $K$ and $N$, 
	i.e. for any ideal $L \unlhd N$ of the same kind the containments $K \subseteq L \subseteq N$ imply either $L = K$ or $L = N$. Maximal 
	left $N$-subloops are defined analogously. 
	A left ideal $K \tensor*[_N]{\unlhd}{} N$ will be called {\em $N$-maximal} if $K$ is 
	a maximal left $N$-subloop. 
	
	Define
	\begin{align*}
		J_2(N) &:= \bigcap \{ K : K \tensor*[_N]{\unlhd}{} N \textrm{ is $N$-maximal} \} \textrm{, }
	\intertext{and}
		R(N) &:= \bigcap \{ I : I \tensor*[_N]{\le}{} N \textrm{ is maximal} \} \textrm{,  }
	\end{align*}
	i.e. $J_2(N)$ is the intersection of all $N$-maximal left ideals, and $R(N)$ is the intersection of all maximal left $N$-subloops. 
	If $N$ is a ring, both definitions coincide with the definition of the Jacobson radical of $N$, hence $J(N) = J_2(N) = R(N)$ 
	for a ring $N$.
	
	For an $N$-maximal left ideal $K \tensor*[_N]{\unlhd}{} N$ the quotient $G:=N/K$ is a nontrivial left $N$-module, which contains no 
	proper nontrivial left $N$-subloops, i.e. $0$ and $G$ are the only left $N$-subloops in $G$ (Proposition~\ref{prop:correspondence}). 
	We will call such left $N$-modules {\em $N$-simple} (as in~\cite[Definition 1.36]{Pilz}). 
	Every $N$-simple left $N$-module $G$ (over unital $N$) is generated by any nonzero $a \in G$, since $Na$ is 
	a nontrivial left $N$-subloop, hence $Na = G$. Also, for an $N$-simple $G$, the kernel of the left $N$-module homomorphism 
	$\phi_a \colon N \to G$, $n \mapsto na$ is an $N$-maximal left ideal, whenever $a \in G$ is nonzero. In fact, every 
	$N$-maximal left ideal arises in this way, and, since 
	\(
		\Ann(G) = \bigcap_{a \in G} \Ann(a) = \bigcap_{a \in G} \ker \phi_a \textrm{, }
	\)
	we can write
	\[
		J_2(N) = \bigcap \{ \Ann(G) : G \textrm{ an $N$-simple left $N$-module}\} \textrm{. }
	\]
	While not as direct as our original definition, this shows that $J_2(N)$ is a two-sided ideal by Corollary~\ref{cor:AnnG} and 
	Lemma~\ref{lem:intersub}. We clearly have 
	$R(N) \subseteq J_2(N)$. It follows from Zorn's lemma (and the fact that $N$ is unital), that every proper left $N$-subloop is contained 
	in a maximal one, hence $R(N)$ is always a proper left $N$-subloop in $N$. On the other hand, $N$ may not have any 
	$N$-maximal left ideals, in which case $J_2(N) = N$. 
	\begin{remark}
		The reader may be wondering, why not simply define a left $N$-module $G$ to be {\em simple} if it contains no proper 
		nontrivial left $N$-submodules. A valid point, with an additional complication. As it turns out, simplicity of $G$ is not enough, 
		one has to require that $G$ is also {\em monogenic}, i.e. there is an $a \in G$ such that $Na = G$. (While $N$-simple left 
		$N$-modules are automatically monogenic, simple left $N$-modules are not.) Then one defines
		\[
			J_0(N) := \bigcap \{ \Ann(G) : G \textrm{ a simple, monogenic left $N$-module}\} \textrm{. }
		\]
		This radical too has an `internal' description
		\[
			J_0(N) = \bigcap \{ (K:N) : K \tensor*[_N]{\unlhd}{} N \textrm{ is maximal} \} \textrm{, }
		\]
		since we can write $G \cong N/K$ for some maximal left ideal $K \tensor*[_N]{\unlhd}{} N$ and 
		$\Ann(G) = \Ann(N/K) = (0:N/K) = (K:N)$ holds. Unlike $J_2(N)$, $J_0(N)$ is different from the intersection 
		of all maximal left ideals, which is often denoted by
		\[
			D(N) := \bigcap \{ K : K \tensor*[_N]{\unlhd}{} N \textrm{ is maximal} \} \textrm{. }
		\]
	\end{remark}
	Moreover, most authors of near-ring literature call simple, monogenic left $N$-modules modules of {\em type 0}, and 
	$N$-simple left $N$-modules modules of {\em type 2}~\cite[Definition 3.5]{Pilz}, \cite[Definition 3.4]{Meldrum}. Of course, there are 
	also left modules of {\em type 1} with their corresponding $J_1(N)$. All three radicals are different for general $N$. 
	(The equality $J_1(N)=J_2(N)$ holds for unital $N$ though.) See~\cite{Pilz} and~\cite{Ramakotaiah} for a precise treatment of 
	those radical-like ideals and left $N$-subloops for a near-ring $N$. We will remain focused on $R(N)$ and $J_2(N)$.
	\begin{definition}
		An element $q \in N$ is {\em quasiregular} if $y = 1 \rdiv q$, the solution of $y + q = 1$, has a left inverse in 
		$N$, i.e. there exists an element $y^\lambda \in N$, such that $y^\lambda y = 1$. A subset $Q \subseteq N$ is {\em quasiregular} 
		if all of its elements are quasiregular.
	\end{definition}
	When $N$ is a near-ring, i.e. $(N,+)$ is a group, our definition of a quasiregular element coincides with~\cite[Definition 1]{Beidleman}, 
	but it is different from~\cite[Definition 5.19]{Maxson}. Quasiregularity in the sense of~\cite{Maxson} for loop near-rings was considered 
	in~\cite{Ramakotaiah}. We note however that~\cite[Definition 4.1]{Ramakotaiah} seems a bit unnatural in the loop near-ring setting, as it 
	considers the left invertibility of $1 + (0\rdiv q)$, which is different from $1\rdiv q$ if $(N,+)$ is a proper loop.
	\begin{remark}
		Assume that an idempotent $e \in N$ is quasiregular. Let $y$ be the solution of the equation $y + e = 1$. Multiplying this 
		equation by $e$ from the right and using right distributivity we obtain $ye + e = e$ or $ye = 0$. Hence 
		$e = y^\lambda y e  = y^\lambda 0 = 0$, i.e. $0$ is the unique quasiregular idempotent. 
		\label{rem:qreg_idem_0}
	\end{remark}
	\begin{lemma}
		\label{lem:qreg_in_R}
		The intersection of all maximal left $N$-subloops $R(N)$ is a quasiregular left $N$-subloop and every quasiregular left ideal 
		$Q \tensor*[_N]{\unlhd}{} N$ is contained in $R(N)$. In particular $D(N) \subseteq R(N)$.
	\end{lemma}
	\begin{proof}
		Pick an $r \in R(N)$, and let $y$ be the solution of the equation $y + r = 1$. 
		Suppose $Ny \neq N$. Then there is a maximal left $N$-subloop $I \tensor*[_N]{\lneq}{} N$ containing the left $N$-subloop $Ny$. 
		Now $r \in R(N) \subseteq I$ implies $1 = y + r \in I$, which is a contradiction. Hence $Ny = N$. In particular $y^\lambda y = 1$ 
		for some $y^\lambda \in N$.
		
		For the second statement, assume that $Q \nsubseteq R(N)$. Then there is a maximal left $N$-subloop $I \tensor*[_N]{\lneq}{} N$, 
		such that $Q \nsubseteq I$, which implies $I + Q = N$, as $I+Q$ is a left $N$-subloop by Proposition~\ref{prop:K+I}. 
		In particular $i + q = 1$ for some $i \in I$ and $q \in Q$. Hence, $i$ has a left inverse $i^\lambda \in N$, which 
		implies $1 = i^\lambda i \in Ni \subseteq I$, a clear contradiction. 		
	\end{proof}
	\begin{corollary}
		Every quasiregular left ideal in $N$ is contained in $J_2(N)$. The radical $J_2(N)$ is quasiregular if and only if 
		$J_2(N) = R(N)$, and $J_2(N)$ is the largest quasiregular ideal in this case.
	\end{corollary}
	\begin{lemma}
		Let $Q \tensor*[_N]{\unlhd}{} N$ be a quasiregular left ideal, $q \in Q$, and $y = 1 \rdiv q$. Then 
		$y$ is invertible, i.e. $y \in U(N)$, and $y^\lambda = y^{-1}$.
		\label{lem:qreg_id-inv}
	\end{lemma}
	\begin{proof}
		We are going to prove that $y^\lambda$ has a left inverse. Since $Q$ is a left ideal, 
		$n + Q = n(y+q) + Q = ny + Q$ holds for all $n \in N$ by Proposition~\ref{prop:substructures}. Picking $n = y^\lambda$ we obtain 
		$y^\lambda + Q = 1 + Q$. Hence, if $x = y^\lambda \ldiv 1$ solves the equation $y^\lambda + x = 1$, then 
		$x \in Q$, so $x$ is quasiregular and $y^\lambda$ has a left inverse.		
	\end{proof}
	\begin{definition}
		A homomorphism $\psi \colon N \to M$, of loop near-rings $N$ and $M$, is {\em local} if $\psi(u) \in U(M)$ implies 
		$u \in U(N)$.
		\label{def:localhom}
	\end{definition}
	The following theorem states that quasiregular ideals in $N$ are precisely the kernels of local homomorphisms.
	\begin{theorem}
		\label{thm:qreg_ur}
		The kernel of a local homomorphism is a quasiregular ideal.
		An ideal $Q \unlhd N$ is quasiregular if and only if the quotient homomorphism $N \to N/Q$ is local.
	\end{theorem}
	\begin{proof}
		Let $\psi \colon N \to M$ be a local homomorphism. Pick any $k \in \ker \psi$ and let $y = 1 \rdiv k$. 
		Then $\psi(y) = \psi(1 \rdiv k) = \psi(1)\rdiv \psi(k) = 1\rdiv 0 = 1$, which is a unit. Hence $y$ is a unit and $k$ is 
		quasiregular. 
		
		It remains to prove the `only if' part of the second statement. If $u + Q$ is invertible in $N/Q$, then there is 
		a $v \in N$, such that $uv + Q = vu + Q  = 1 + Q$. Hence $uv + p = 1$ and $vu + q = 1$ for some $p, q \in Q$. 
		Since $Q$ is a quasiregular (left) ideal it follows from Lemma~\ref{lem:qreg_id-inv} that $uv \in U(N)$ and $vu \in U(N)$, 
		which implies $u \in U(N)$.		
	\end{proof}
	\begin{remark}
	In~\cite{FraPav} and~\cite{FraPav2} local homomorphisms were called {\em unit-reflecting} homomorphisms. The author is grateful to the referee for making him 
	aware that `local homomorphism' is the accepted term in ring theory. Local homomorphisms between rings in the generality of 
	Definition~\ref{def:localhom} have already been used in~\cite{CampsDicks} and~\cite{FacchiniHerbera}.
	\end{remark}
\section{Local loop near-rings} 
	\label{sect:local}
	Local near-rings were introduced by Maxson in~\cite{Maxson}. His main definition is different from ours below, but equivalent 
	to it in case $N$ is a near-ring, see~\cite[Theorem 2.8]{Maxson}. Our discussion will be restricted to unital, zero-symmetric loop 
	near-rings.
	\begin{definition}
		A loop near-ring $N$ is {\em local} if it has a unique maximal left $N$-subloop.
	\end{definition}
	For $N$ local, we will usually denote the unique maximal left $N$-subloop by $\m$, and also write $(N, \m)$ to emphasize the role 
	of $\m$.
	
	In a local loop near-ring $(N,\m)$, the elements of $\m$ do not have left inverses, since $\m$ is a proper left $N$-subloop in $N$. 
	On the other hand, for any $u \in N \setminus \m$ we have $N u = N$, hence $u^\lambda u = 1$ for some $u^\lambda \in N$. Every 
	element not contained in $\m$ is left invertible. 
	
	Suppose that some $n \in \m$ is right invertible, i.e. there is an $n^\rho \in N$ such that $n n^\rho = 1$. As $n^\rho n \in \m$, it 
	follows that $1\rdiv n^\rho n \notin \m$, since $1 \notin \m$. Let $u$ be a left inverse of $1\rdiv n^\rho n$. Then 
	$n^\rho = u (1 \rdiv n^\rho n) n^\rho = u (n^\rho \rdiv n^\rho) = u 0 = 0$, a contradiction. Elements of $\m$ 
	do not even have right inverses.
	\begin{proposition}
		If $(N,\m)$ is local, then $N$ is the disjoint union $\m \cup U(N)$.
		\label{prop:local}
	\end{proposition}
	\begin{proof}
		It follows from above discussion that every $u \in N \setminus \m$ has a left inverse $u^\lambda \in N \setminus \m$, hence 
		$N \setminus \m \subseteq U(N)$. The reverse inclusion is clear, and we can conclude $N \setminus U(N) = \m$ or 
		$N = \m \cup U(N)$.		
	\end{proof}
	\begin{theorem}
		\label{thm:local}
		Let $N$ be a loop near-ring, such that $J_2(N) \neq N$, i.e. $N$ has at least one $N$-maximal left ideal. The following 
		properties are then equivalent:
	\begin{enumerate}[(a)]
			\item $(N,\m)$ is local, i.e. $N$ has a unique maximal left $N$-subloop $\m \tensor*[_N]{\lneq}{} N$.
			\item $N$ has a unique $N$-maximal left ideal $K \tensor*[_N]{\unlhd}{} N$ and this $K$ is quasiregular.
			\item $J_2(N)$ is quasiregular and $N/J_2(N)$ is a loop near-field.
			\item $N \setminus U(N)$ is an ideal in $N$.
			\item $N \setminus U(N)$ is a subloop in $(N,+)$.
			\item $m+n \in U(N)$ implies $m \in U(N)$ or $n \in U(N)$.
		\end{enumerate}
		Moreover, in any of the above cases the equalities 
		$\m = R(N) = J_2(N) = N \setminus U(N)$ hold.
	\end{theorem}
	\begin{proof}
		(a) $\Rightarrow$ (b): Every $N$-maximal left ideal $K$ is contained in $\m$. Since $K$ is also maximal as a left $N$-subloop, 
			it follows $K = \m = R(N)$. Quasiregularity now follows from Lemma~\ref{lem:qreg_in_R}.
		
		(b) $\Rightarrow$ (c): If $K \tensor*[_N]{\unlhd}{} N$ is the unique $N$-maximal left ideal, $J_2(N) = K$ and, since $K$ is 
			quasiregular, $J_2(N)$ is quasiregular. Also, $K$ is a maximal left $N$-subloop, so $Nu = N$ for every $u \in N \setminus K$, 
			in particular $u^\lambda u = 1$ for some $u^\lambda \in N$. Hence, every nonzero class $u + K \in N/K$ 
			has a left inverse $u^\lambda + K \in N/K$, which implies that $N/K \setminus \{K\}$ is a group with respect to multiplication.
			
		(c) $\Rightarrow$ (b): If $N/J_2(N)$ is a loop near-field, $N/J_2(N)$ is $N$-simple when viewed as a left $N$-module, hence 
			$J_2(N)$ is an $N$-maximal left ideal, which is clearly unique.
			
		(b) $\Rightarrow$ (a): Suppose $K$ is the unique $N$-maximal left ideal in $N$. Then $R(N) \subseteq J_2(N) = K$. Moreover, 
		 $K \subseteq R(N)$ by Lemma~\ref{lem:qreg_in_R}, since $K$ is quasiregular. Hence $K = R(N) = \m$.
			
		(c) $\Rightarrow$ (d): By Theorem~\ref{thm:qreg_ur} an element $n \in N$ is invertible if and only if $n + J_2(N)$ is invertible in 
			$N/J_2(N)$. By (c), $U(N/J_2(N)) = N/J_2(N) \setminus \{J_2(N)\}$, hence $U(N) = N \setminus J_2(N)$ or 
			$N \setminus U(N) = J_2(N)$, which is an ideal.
			
		(d) $\Rightarrow$ (e) $\Rightarrow$ (f): These are tautologies.
		
		(f) $\Rightarrow$ (c): Since $J_2(N) \neq N$, we must have $J_2(N) \subseteq N \setminus U(N)$. Let $y$ solve the equation 
			$y + j = 1$ for $j \in J_2(N)$. Since $j \notin U(N)$, (f) implies $y \in U(N)$, hence $J_2(N)$ is quasiregular.
			
			Take any $v \notin J_2(N)$ and let $K$ be an $N$-maximal left ideal, which does not contain $v$. Note that $K + Nv$ is 
			a left $N$-subloop by Proposition~\ref{prop:K+I}, and, since $K$ is a maximal left $N$-subloop, $K + Nv = N$. Therefore 
			$k + uv = 1$ for some $k \in K$ and $u \in N$. Now, by (f), $k \notin U(N)$ implies $uv \in U(N)$. We have just shown that 
			every nonzero class $v + J_2(N)$ has a left inverse, hence $N/J_2(N) \setminus \{J_2(N)\}$ is a group with respect 
			to multiplication.			
	\end{proof}
	\begin{remark}
		We note, without proof, that in a local loop near-ring $(N,\m)$ with $J_2(N) \neq N$, all radical-like ideals and subsets are equal, 
		i.e.
		\[
			\m = J_0(N) = D(N) = R(N) = J_1(N) = J_2(N) \textrm{. }
		\]
	\end{remark}
	As is the case for local rings, (zero-symmetric) local loop near-rings cannot contain proper nontrivial idempotents.
	\begin{lemma}
		If $e$ is an idempotent in a local loop near-ring $(N,\m)$, then either $e = 0$ or $e = 1$.
	\end{lemma}
	\begin{proof}
		By Proposition~\ref{prop:local} either $e \in \m$ or $e \in U(N)$. By Lemma~\ref{lem:qreg_in_R} the $N$-subloop $\m = R(N)$ is 
		quasiregular, so $0$ is the only idempotent in $\m$ (Remark~\ref{rem:qreg_idem_0}). On the other hand, the identity 
		$1$ is the only idempotent in the multiplicative group $U(N)$.		
	\end{proof}
	\begin{lemma}
		\label{lem:pre_local}
		Let $\psi \colon N \to M$ be a local homomorphism of loop near-rings. Assume $J_2(N) \neq N$ and $J_2(M) \neq M$. 
		If $M$ is local, then $N$ is also local.
	\end{lemma}
	\begin{proof}
		Pick $m, n \in N$, such that $m+n \in U(N)$. Then $\psi(m+n) = \psi(m) + \psi(n)$ is a unit in $M$. By 
		Theorem~\ref{thm:local} either $\psi(m)$ or $\psi(n)$ is a unit in $M$. Now, $\psi$ is local, so either $m$ or $n$ 
		is a unit in $N$, which shows that $N$ is local by another use of Theorem~\ref{thm:local}.		
	\end{proof}
	If $M$ is a ring, the assumption $J_2(N) \neq N$ in Lemma~\ref{lem:pre_local} is unnecessary.
	\begin{lemma}
		Assume $N$ is a loop near-ring, $R$ a nontrivial ring, and $\psi \colon N \to R$ a loop near-ring homomorphism. Then 
		$J_2(N)$ is a proper ideal in $N$.
	\end{lemma}
	\begin{proof}
		Since $\psi(0) = 0$ and $\psi(1) = 1$, $\im \psi$ is a nontrivial subring in $R$, so  
		$J(\im \psi) = J_2(\im \psi) \neq \im \psi$. Note that the preimage of a maximal left ideal 
		$K \tensor*[_{\im \psi}]{\unlhd}{} \im \psi$ is an $N$-maximal left ideal 
		$\psi^{-1}(K) \tensor*[_N]{\unlhd}{} N$. This can be restated as $J_2(N) \subseteq \psi^{-1}(J(\im \psi))$, and, 
		since $\psi^{-1}(J(\im \psi)) \neq N$, $J_2(N) \neq N$.		
	\end{proof}
	\begin{corollary}
		Assume there is a local homomorphism $\psi \colon N \to R$ from a loop near-ring $N$ to a local ring $R$. 
		Then $J_2(N) \neq N$, $N$ is a local loop near-ring, and $N/J_2(N)$ is a division ring.
	\end{corollary}
\bibliographystyle{plain}
\bibliography{ref}
\end{document}